\documentclass[12pt, a4paper]{amsart}

\usepackage{a4wide}
\usepackage[english]{babel}
\usepackage[T2A]{fontenc}
\usepackage[utf8]{inputenc} 
\usepackage{amsfonts}
\usepackage{amssymb, amsthm, amscd}
\usepackage{amsmath}
\usepackage{mathtools}
\usepackage{needspace}
\usepackage{etoolbox}
\usepackage{mathabx}
\usepackage{lipsum}
\usepackage{comment}
\usepackage{cmap}
\usepackage[pdftex]{graphicx}
\usepackage[unicode]{hyperref}
\usepackage[matrix,arrow,curve]{xy}
\usepackage[usenames,dvipsnames]{xcolor}
\usepackage{colortbl}
\usepackage{textcomp}
\usepackage{cite}
\usepackage{euscript}
\usepackage{ upgreek }

\makeatletter
\def\@settitle{\begin{center}
		\baselineskip14\p@\relax
		\bfseries
		\LARGE
		\@title
	\end{center}
}
\makeatother

\newtheorem{theorem}{Theorem}[section]
\newtheorem{lemma}[theorem]{Lemma}

\theoremstyle{definition}
\newtheorem{definition}[theorem]{Definition}

\renewcommand{\geq}{\geqslant}
\renewcommand{\leq}{\leqslant}
\newcommand{\R}{\mathbb{R}}
\newcommand{\Z}{\mathbb{Z}}

\title{Frame set for shifted sinc-function}
\author{Yurii Belov and Andrei V. Semenov}

\thanks{The research was supported by Russian Science Foundation Grant 22-11-00071. The second author was supported by the Ministry of Science and Higher Education of the Russian Federation, agreement no. 075-15-2022-287. Also the second author is a winner of Young Russian Mathematics award and he is grateful to its jury and sponsors.}

\address{Andrei V. Semenov:
St. Petersburg State University, Universitetskaya emb. 7/9,
Saint Petersburg, 199034 Russia}

\email{asemenov.spb.56@gmail.com}

\address{Yurii S. Belov:
Department of Mathematics and Computer Sciences, 
St. Petersburg State University, 14th Line V.O., 29B, 
Saint Petersburg 199178 Russia \\ Yanqi Lake Beijing Institute of Mathematical Sciences and Applications.}

\email{j\_b\_juri\_belov@mail.ru}

\begin{document}
\maketitle


\begin{abstract}
We prove that frame set $\mathcal{F}_g$ for imaginary shift of sinc-function  
$$g(t)=\frac{\sin\pi b(t-iw)}{t-iw}, \quad b,w\in\mathbb{R}\setminus\{0\}$$ 
can be described as $\mathcal{F}_g=\{(\alpha,\beta): \alpha\beta\leq 1, \beta\leq|b|\}.$ \\
In addition, we prove that $\mathcal{F}_g=\{(\alpha,\beta): \alpha\beta\leq 1 \}$  for window functions $g$ of the form $\frac{1}{t-iw}(1-\sum\limits_{k=1}^{\infty}a_ke^{2\pi i b_k t})$, such that $\sum_{k\geq 1}|a_k|e^{2\pi|w|b_k}<1$, $wb_k<0$. 
\end{abstract}

\section{Introduction}

One of the main questions in time-frequency analysis is to find a representation of arbitrary $f\in L^2(\mathbb{R})$ as a sum of well localized functions in time-frequency plane.
We can achieve this goal using Gabor systems. With any function $g\in L^2(\mathbb{R})$ we associate its time-frequency shifts $\pi_{x,w}g$, where $\pi_{x,w}g(t)=e^{2\pi i \omega t}g(t-x)$ for any $(x,\omega) \in\mathbb{R}^2$.

\begin{definition} Let $\alpha,\beta>0$. Put
$$\mathcal{G}(g;\alpha,\beta)=\{\pi_{\alpha n, \beta m} g\}_{m,n\in\mathbb{Z}}.$$
We say that system $\mathcal{G}(g;\alpha,\beta)$ is a Gabor system over the lattice $\alpha \Z\times\beta\mathbb{Z}$ with window $g$. 
\end{definition}

Gabor analysis studies the
question under which conditions on $\alpha$, $\beta$ and $g$ there exist constants $A,B>0$,
such that {\it frame inequality} holds,

\begin{equation}
A\|f\|^2_2\leq \sum_{m,n}|(f, \pi_{\alpha n, \beta m}g)|^2\leq B\|f\|^2_2, \quad f\in L^2(\mathbb{R}).
\label{frameineq}
\end{equation}

In particular, frame inequality implies existence of dual window $\gamma\in L^2(\mathbb{R})$ with two time-frequency representations, see \cite{Gro},

\begin{equation}
f=\sum_{m, n}(f, \pi_{\alpha n, \beta m}g)\pi_{\alpha n, \beta m}\gamma=\sum_{m, n}(f, \pi_{\alpha n, \beta m}\gamma)\pi_{\alpha n, \beta m}g.
\end{equation}

\subsection{Gabor frames} If \eqref{frameineq} holds the system $\mathcal{G}(g;\alpha,\beta)$ is called a {\it Gabor frame}. The set 
$$\mathcal{F}_g =\{(\alpha,\beta): \mathcal{G}(g;\alpha,\beta)  \text{ is a Gabor frame}\}$$
is called {\it the frame set } of $g$.

If $\alpha\beta>1$ then Gabor system is not a frame, see \cite{Gro}. Note that complete characterization of a frame set for the critical hyperbola $\alpha\beta=1$ can be given in terms of {\it Zak transform} $\mathcal{Z}g$ of a window $g$ (see e.g. \cite[Ch. 8]{Gro}). But for $\alpha\beta<1$ a frame set $\mathcal{F}_g$ may have complicated structure even for elementary functions $g$  (see, for example, \cite{BC, Jans, Jans4, DS}). 

The complete description of $\mathcal{F}_g$ has been obtained for the Gaussian $e^{-x^2}$ (see \cite{L,S,SW}), truncated exponential function $\chi_{x>0}e^{-x}$ as well as symmetric exponential function $e^{-|x|}$ (see \cite{Jans2,Jans3}), and for the hyperbolic secant $\frac{1}{e^x+e^{-x}}$ (see \cite{JansStr}). Despite numerous attempts very little progress has been done before 2011. A breakthrough
was achieved in \cite{Gro1} and later in \cite{Gro2} where the authors considered the class of totally
positive functions of finite type and (by using another approach) Gaussian totally positive functions of finite type. Recently first author described frame set for {\it rational functions of Herglotz type} together with A. Kulikov and Yu. Lyubarskii (see \cite{BKL1}).

\subsection{Sum of modulations of Cauchy kernels}
Our goal is to describe frame set for a new class of window functions $g(t)$. We will consider functions of the form
$$g(t)=\frac{1}{t-iw}\biggl{(}\sum \limits_{k=0}^{\infty}a_ke^{2\pi  i b_k t}\biggr{)}, \text{ where }  \sum_k|a_k|<\infty,\quad b_k,w\in\mathbb{R}.$$

For some $P(t)$ and $\alpha\beta<1$ the Gabor system $\mathcal{G}(g;\alpha,\beta)$ is not a frame. In particular, if $P(iw)=0$ and $\sup_k|b_k|<\infty$, then Fourier transform of ${g}$ has a compact support and so the Gabor system $\mathcal{G}(g;\alpha,\beta)$ is not a frame for sufficiently big $\beta$.

On the other hand, we are able to find a some subclass  with maximal possible frame set.

\begin{theorem} Let $g(t) = {\frac{P(t)}{t - iw}}$, where $P(t) = \sum \limits_{k=0}^{\infty} a_k e^{2\pi i b_k t}$ such that $w<0$, $a_0=1, b_0=0$, $b_k>0$ for $k\geq 1$. If
$$\sum_{k=1}^\infty|a_k|e^{2\pi |w|b_k}<1,$$
 then
$\mathcal{F}_g=\{(\alpha,\beta) \mid \alpha\beta\leq 1\}$.
\label{mainth}
\end{theorem}

Of course this result is still true for finite sums (i.e. trigonometric polynomials $P$). The proof based on properties of circular shift operator  in Paley-Wiener spaces. In particular, this technique allows us to completely describe frame set for imaginary shift of sinc-function.

\begin{theorem} Let
$$g(t)=\frac{\sin\pi b(t-iw)}{t-iw}, \quad b>0, w<0.$$
Then $\mathcal{F}_g=\{(\alpha,\beta) \mid \alpha\beta\leq 1, \beta\leq b\}$.
\label{mainth2}
\end{theorem}

We normalize the Fourier transform as 
$$\hat{g}(\xi)=\int_\mathbb{R}g(t)e^{-2\pi i t\xi}dt.$$

Note that the  Fourier transform of $\frac{\sin\pi b(t-iw)}{t-iw}$ is supported on the interval $[-b\slash2,b\slash2]$. Since $\mathcal{G}(g, \alpha, \beta)$ is a frame if and only if $\mathcal{G}(\hat{g}, \beta, \alpha)$ is a frame, we have $\beta\leq b$ is necessary condition for inclusion $(\alpha,\beta)\in\mathcal{F}_g$.

\smallskip

It is worth note that non-shifted sinc-function, $w=0$ is a Fourier transform of characteristic function of an interval and so the frame set has very complicated structure, see \cite{DS}. 

\smallskip

Our technique allows us to find an explicit expression for the frame operator $S^{\alpha,\beta}$ corresponding to sinc-function, see \eqref{sincfr} For arbitrary window of the form $P(t)\slash (t-iw)$ see \eqref{pfr}. 

\smallskip

{\bf Organization of the paper and notations.} In Section \ref{MC} we prove the main criterion. In Section \ref{MR} we prove Theorem \ref{mainth}. In Section \ref{SF} we prove Theorem \ref{mainth2}. In Section \ref{FO} we find an explicit expressions for frame operators.

\smallskip

Throughout this paper, $U(z)\lesssim V(z)$ (equivalently $V (z) \gtrsim U(z)$) means that there exists a constant $C$
such that $U(z) \leq CV (z)$ holds for all $z$ in the set in question, which may be either a Hilbert
space or a set of complex numbers (or a suitable index set). We write $U(z) \asymp V (z)$ if both
$U(z)\lesssim V (z)$ and $V (z)\lesssim U(z)$ holds.

\section{Main Criterion \label{MC}}

We define Paley-Wiener space $\mathcal{P}W_a$ as an inverse Fourier transform of a space $L^2(0,a)$,
$$\mathcal{P}W_a=\biggl{\{}f \mid f(t)= \int_0^a\varphi(t)e^{2\pi i zt}dt, \quad \varphi\in L^2(0,a)\biggr{\}}.$$

We need the following {\it circular shift operator} $L_b: L^2(0, \alpha^{-1})\longrightarrow L^2(0,\alpha^{-1})$:
$$L_{b} g = \begin{cases}
{g}(x-b), & x \in [b, \alpha^{-1}],\\
{g}(x - b + \alpha^{-1}), & x \in [0,b],
\end{cases}$$
for $b\in[0,\alpha^{-1})$.
Also put $M_f g = fg$ for $g\in L^2(0,\alpha^{-1})$. 

Define operator $A_k \colon L^2(0, \alpha^{-1}) \longrightarrow L^2(0, \alpha^{-1})$ by the rule
$$A_k f(t):= M_{e^{-2\pi  wt}} L_{b_k^{\star}} M_{e^{2\pi  wt}} f(t) \text{ for any } f \in L^2(0, \alpha^{-1}),$$ 
where 
$b^{\star}_k= b_k \mod {{{1} \over {\alpha}}}$, i.e.
    $$A_k f(t) = \begin{cases}
        e^{-2\pi w b_k^*} f(t-b_k^*), & b_k^* < t \leq \alpha^{-1} \\
        e^{-2\pi w(b_k^* - \alpha^{-1})} f(t - b_k^* + \alpha^{-1}), & 0 \leq t \leq b_k^*.
    \end{cases}$$

Now for the function $$g(z) = (1-e^{\frac{2\pi i z}{\alpha}}) \sum_{n} {{c_n} \over {z-\alpha n}} \in \mathcal{P}W_{1 / \alpha}$$
let us consider operator $\widecheck{L}_b$ given by the formula
$$g(z) \mapsto \big(1-e^{\frac{2\pi i z}{\alpha}} \big) \sum_{n} {{c_n e^{ 2\pi i \alpha b n}} \over {z-\alpha n}}.$$
Also consider operator $S_w$ in $\mathcal{P}W_{1/\alpha}$ defined by the rule $f \mapsto f(z-iw)$.
\begin{lemma}\label{lem:1}
    We have $\mathcal{F} L_b = \widecheck{L}_b$ and $\mathcal{F} M_{e^{-2\pi  wt}} = S_w$.
\end{lemma}
\begin{proof}

It is easy to check that $\widecheck{g}(t)=\sum_n c_n e^{2 \pi i \alpha n t}$, $t\in[0,\alpha^{-1})$. On the other hand,

$$(\widecheck{L}_bg)^{\widecheck{}}= \sum_n c_n e^{2\pi i \alpha b n }e^{2 \pi i \alpha n t}= \sum_n c_n e^{2 \pi i \alpha n (t+b)}.$$
So, operator $\widecheck{L}_{b}$ is a circular shift in Fourier domain and we get $\mathcal{F} L_b = \widecheck{L}_b$.

The second statement follows from substitution in the formula for inverse Fourier transform.

\end{proof}

\begin{lemma}\label{lem:2} Let 
$$h(z)= \big( 1-e^{\frac{2\pi i (z-iw)}{\alpha}}  \big) \bigg( \sum\limits_n {{c_{n}  e^{2\pi i b \alpha n}} \over {z-\alpha n - iw}} \bigg)$$
and let 
$$h^{(0)}(z) = \big( 1-e^{\frac{2\pi i (z-iw)}{\alpha}}  \big) \bigg( \sum\limits_n {{c_{n}} \over {z-\alpha n - iw}} \bigg).$$
Then 
$$h(z) =  S_w \widecheck{L}_{b} S_{-w} h^{(0)}(z).$$
\end{lemma}
\begin{proof}
   It is obvious by definition.
\end{proof}
Now we are ready to formulate the main criteria for frame inequality.

\begin{theorem} Let $P(t)=\sum \limits_{k=0}^{\infty} a_k e^{2\pi i b_k t}$, such that $ \sum_k|a_k| <\infty$, $w<0$ and $b_k\geq0$ for each $k$. Let $g(t)=\dfrac{P(t)}{t-iw}$. Then the Gabor system $\mathcal{G}(g;\alpha, 1)$ is a frame in $L^2(\mathbb{R})$ if and only if
\begin{equation}
    \sum_{m\in\mathbb{Z}}\int_0^{\frac{1}{\alpha}}\biggl{|}\sum_{k=0}^{\infty} a_k A_k G(t+m+b_k)\biggr{|}^2dt\asymp\|G\|^2_2
\label{maineq}
\end{equation}
 for any $G \in L^2(\R)$.
\end{theorem}

Hereinafter we consider function $G(t+m+b_k)$ as an element of the space $L^2(0,\alpha^{-1})$.

\begin{proof}

Note that 
$$g_{m,n} (t) = \pi_{\alpha n, m} g(t) = e^{2\pi im t} g(t-\alpha n) = {{e^{2\pi im t} \sum\limits_{k=0}^{\infty}  a_ke^{2\pi ib_k(t-\alpha n)}} \over {t - \alpha n - iw}}.$$ 

It is immediate that right hand side of \eqref{maineq} does not exceed $C\|G\|^2$ for some $C>0$. So it is sufficient to prove  $$\sum_{m,n}|(f,g_{m,n})|^2\asymp \sum_{m\in\mathbb{Z}}\int_0^{\frac{1}{\alpha}}\biggl{|}\sum_{k=0}^{\infty}a_k A_{k} f(t+m+b_k)\biggr{|}^2dt, \quad f\in L^2(\mathbb{R}).$$ 
Now consider
$$|(f, g_{m,n})|^2 = |(g_{m,n}, \overline{f})|^2 = \int_{\R} \overline{f(t)} e^{2\pi i mt} \sum_{k=0}^{\infty} {{a_k e^{2\pi i b_k t} e^{-2\pi i b_k \alpha n}} \over {t -\alpha n - iw}} dt.$$


Observe that $\sum \limits_{m,n}|(f,g_{m,n})|^2 = \sup\limits_{\|c_{m,n} \|\leq 1} \sum\limits_{m, n \in \mathbb{Z}} c_{mn} (f, g_{m,n})$, where sequences $c_{mn}$ run over $\ell^2(\Z \times \Z)$. Now we have
$$\sum_{m,n}|(f,g_{m,n})|^2 = \sup_{\|c_{mn}\| \leq 1} \left| \sum_{m, n \in \mathbb{Z}} c_{mn} (f, g_{m,n}) \right| = $$
$$=\sup_{\| c_{mn}\| \leq 1} \int_{\R} \overline{f(t)} \sum_m e^{2\pi i mt}  \bigg( \sum_{k=0}^{\infty} \sum_n {{c_{mn} a_k e^{2\pi i b_k(t - \alpha n)}} \over {t-\alpha n - iw}} \bigg) dt=$$
$$=\sup_{\| c_{mn}\| \leq 1} \int_{\R} {{\overline{f(t)}} \over {1-e^{\frac{2\pi i (t-iw)}{\alpha}}  }} \sum_m e^{2\pi i mt} \big( 1-e^{\frac{2\pi i (t-iw)}{\alpha}}  \big) \bigg( \sum_{k=0}^{\infty} \sum_n {{c_{mn} a_k e^{2\pi i b_k(t - \alpha n)}} \over {t-\alpha n - iw}} \bigg) dt.$$

Put 
$$h_m^{(k)}(t)= a_k e^{2\pi i b_k t} \big( 1-e^{\frac{2\pi i (t-iw)}{\alpha}}  \big) \bigg( \sum\limits_n {{c_{mn}  e^{-2\pi i b_k \alpha n}} \over {t-\alpha n - iw}} \bigg), \quad k \geq 0,$$
and let $h_m^{(0)}(t)$ be the function defined by the rule
    $$h_m^{(0)}(t) = \big( 1-e^{\frac{2\pi i (t-iw)}{\alpha}}  \big) \bigg( \sum\limits_n {{c_{mn}} \over {t-\alpha n - iw}} \bigg).$$
Also define a function $F(t)$ by the formula
$$F(t) = {{f(t)} \over {1-e^{-\frac{2\pi i (t-iw)}{\alpha}}  }}.$$ 
We have $h^{(k)}_m \in \mathcal{P}W_{1/\alpha}$. 

Now applying Lemma \ref{lem:2} we have 
$$h_m^{(k)}(z) =  a_k e^{2\pi i b_k t}  S_w \widecheck{L}_{-{b}^{\star}_k} S_{-w} h_m^{(0)}(z).$$ 
 Hence 
$$\widecheck{h}_m^{(k)}(t) = a_k M_{e^{-2\pi  wt}} L_{-{b}^{\star}_k} M_{e^{2\pi  wt}} \widecheck{h}_m^{(0)}(t-b_k)$$
by Lemma \ref{lem:1}. So one can obtain
$$\sum_{n,m} |(f, g_{m,n})|^2 \asymp \sup_{\| c_{mn}\| \leq 1} \sum_m \int_{\R} \overline{F(t) e^{-2\pi i mt}} \bigg( \sum_{k=0}^{\infty} h_m^{(k)}(t) \bigg) dt = $$
$$=\sup_{\| c_{mn}\| \leq 1} \sum_m \int_{0}^{{{1} \over {\alpha}}} \overline{\widecheck{F}(t+m)} \cdot \bigg( \sum_{k=0}^{\infty} a_k  M_{e^{-2\pi wt}} L_{-b^{\star}_k} M_{e^{2\pi wt}}  \widecheck{h}_m^{(0)}(t-b_k)\bigg) dt.$$
 Here the last equality holds by Parseval identity. We have
$$\sup_{\| c_{mn}\| \leq 1} \sum_m \int_{\R} \sum_{k=0}^{\infty} a_k   M_{e^{-2\pi wt}} L_{b^{\star}_k} M_{e^{2\pi wt}}B_{b_k} \overline{\widecheck{F}(t+m)} \cdot \check{h}_m^{(0)}(t)dt,$$
where $B_k$ stands for shift operator on $L^2(\R)$, $B_k(f)(t) = f(t+k)$. Let $G(t) = \overline{F(t)}$ and observe that when $\{c_{mn}\}_{m,n}$ runs over $\ell^2(\mathbb{Z}^2)$, the function $\check{h}_m^{(0)}(t)$ runs over $L^2[0, 1/ \alpha]$. Note that $A_k = M_{e^{-2\pi wt}} L_{b^{\star}_k} M_{e^{2\pi wt}}$. Finally we obtain
$$\sum_{n,m} |(f, g_{m,n})|^2 \asymp \sup_{\| c_{mn}\| \leq 1} \sum_m \int_{\R}  \sum_{k=0}^{\infty} \overline{a_k  A_k \left( \chi_{[0, 1/\alpha]}(t) G(t+m+b_k) \right)} \check{h}_m^{(0)}(t) dt \asymp$$
$$\asymp \sum_m \int_{0}^{\frac{1}{\alpha}}  \left| \sum_{k=0}^{\infty} a_k  A_k \left(\chi_{[0, 1/\alpha]} (t) G(t+m+b_k) \right) \right|^2 \ dt.$$
Here the upper bound follows from Cauchy inequality and the lower bound can be obtained by substituting sequence $c_{mn}$ corresponding to the function
$$\check{h}_m^{(0)}(t) := \sum_{k=0}^{\infty} a_k  A_k \left( \chi_{[0, 1/\alpha]} (t) G(t+m+b_k) \right),$$
which is in $L^2(0, \alpha^{-1})$ since it can be represented as a series $\sum_k a_k e^{2\pi |w| b_k^*} f_k$, where $\|f_k \|_2$ bounded uniformly by $\|G\|_2$ and each $e^{2\pi |w| b_k^*}$ bounded by $e^{2\pi |w| \alpha^{-1}}$.
\end{proof}


\section{Proof of Theorem \ref{mainth} \label{MR}}

In this section we prove Theorem \ref{mainth}.
By natural re-normalization we have: 
$$\mathcal{G}(g;\alpha,\beta) \text{ is a frame in } L^2(\mathbb{R}) \text{ if and only if } \mathcal{G}(g(\cdot\slash\beta);\alpha\beta,1) \text{ is a frame in } L^2(\mathbb{R}).$$
So, it is sufficient to prove frame inequality for $\beta=1$, $\alpha\leq 1$ and all re-normalizations of parameters $(aw,b_k\slash a), a>0$. 
\smallskip

Let $\beta=1$. We will use main criterion. The upper inequality is straightforward since all the operators $A_k$ are bounded.
In order to prove lower inequality we consider operators $A_k$ in details.

\smallskip

Note that we have
\begin{equation}
A_{k}G(t+m+b_k)= \begin{cases} e^{-2\pi w b^*_k}G(t+m+b_k-b_k{*}) &  \text{ on } (b_k^*, \alpha^{-1}],\\
e^{-2\pi w (b_k^{\star} - {{1} \over {\alpha}})}G(t+m+b_k-b_k{*}+\alpha^{-1}) & \text{ on }  [0, b_k^*].\end{cases}
\label{Ak}
\end{equation}
Since $b_k-b^*_k = \frac{l}{\alpha}$ for some non-negative integer $l$
we conclude that the sum $\sum \limits_{k=0}^{\infty} a_kA_kG(t+m+b_k)$
contains values of $G$ only in one arithmetic progression 
$$\biggl{\{}G\biggl{(}\xi+\frac{l}{\alpha}\biggr{)}\biggr{\}}_{l\in\mathbb{Z}}.$$

For each $\xi \in(0,\alpha^{-1})$ we consider all solutions of the equation $t+m=\xi \text{ mod } \frac{1}{\alpha}$, where $t\in [0,\alpha^{-1})$ and $m\in\mathbb{Z}$. It is countable set since for any $m \in \Z$ there is only finitely many such $t$. Denote by $T_m$ the set of all $t \in (0, \alpha^{-1})$ such that $t+m=\xi \text{ mod } \frac{1}{\alpha}$ and observe that $\# T_m$ bounded uniformly with respect to $m$. Let us consider the operator $\mathcal{L}_\xi: L^2(\mathbb{R})\mapsto \ell^2$ by the rule
\begin{equation}
\mathcal{L}_\xi G=\biggl{\{}\sum_{k=0}^{\infty} a_kA_kG(t+m+b_k)\biggr{\}}_{m \in \Z, \ t \in T_m}
\label{leq}
\end{equation}
Observe that the definition is correct since $\|A_k\|_2$ uniformly bounded with respect to $k$  by some constant $C$ and so for a fixed $k$ we have
$$\sum_{(m,t)} | A_k G(t + m + b_k) |^2 \leq C \|G\|_2^2 = C_1$$
for almost all $\xi$.
Now 
$$\sqrt{\sum_{(m,t)} \left| \sum_{k=0}^{\infty} a_k A_k G(t + m + b_j) \right|^2 }\leq \sum_k \left| \sum_{(m,t)} a_k^2 (A_k G(t+m+b_k))^2\right|^{1/2} \leq C_1 \sum_k |a_k| < \infty.$$
It remains to prove that for some $\varepsilon>0$,
$$\int_0^{1\slash\alpha}\|\mathcal{L}_\xi G\|_{\ell^2}^2d\xi\geq\varepsilon\|G\|^2,\quad G\in L^2(\mathbb{R}).$$
If we consider $\mathcal{L}_\xi$ as a matrix operator from $\ell^2$ to $\ell^2$ it is sufficient to prove that for any sequence $\{x_n\}$,
$$\|\mathcal{L}_\xi (x_n)\|_{\ell^2}\geq \varepsilon \|x_n\|_{\ell^2}.$$
We prove the desired inequality for the operator $\mathcal{L'_\chi}$,

\begin{equation}
\mathcal{L'}_\xi G=\biggl{\{}\sum_{k=0}^{\infty} a_kA_kG(t+m+b_k)\biggr{\}}, \quad t+m = \xi \text{ mod } \frac{1}{\alpha}, \quad t\in [0,1), m\in\mathbb{Z}.
\label{lseq}
\end{equation}
The only difference with  $\mathcal{L_\xi}$ is that we consider $t<1$ and hence for each
$\xi\in\mathbb{R}$ we have a unique representations $\xi=t+m$, where $m\in\mathbb{Z}$. We have 
$$\|\mathcal{L}'_\xi (x_n)\|\leq\|\mathcal{L}_\xi (x_n)\|,$$
since a set $T_m'$ of solutions $t+m = \xi \text{ mod } \alpha^{-1}$ for $t \in [0,1]$ is strictly smaller then a set $T_m$ and so in the evaluation of $ \| \mathcal{L}' (x_n)\|_2 $ we only delete some positive summands which are in the evaluation of $\| \mathcal{L} (x_n) \|_2$. 
So it is sufficient to prove that 
$\|\mathcal{L}'_\xi (x_n)\|_{\ell^2}\geq \varepsilon \|x_n\|_{\ell^2}.$
\smallskip
Now assume for a while that $b_k>\alpha^{-1}$ for $k\geq1$. Then from \eqref{Ak} and the observation $a_0=1$, $b_0=0$ it follows that the matrix operator $\mathcal{L}'_\xi$ may have the form 
$$\mathcal{L}'_{\xi}=\begin{pmatrix} & & ... & &\\
... & 1 & 0 &\star&\star &0 &..\\
... & 0 & 1 & \star & 0 &...\\
... & 0 & 0 & 1 & \star & 0 &...\\
... & 0 & 0 & 0 &  1 & 0 &...\\
\end{pmatrix}.
$$
More precisely, for a fixed row each index $k\geq1$ gives us either $a_ke^{-2\pi w b^{*}_k}$ or 
$a_ke^{-2\pi w(b^{*}_k-\alpha^{-1})}$ in the place $l(k)$, where $l(k)=\alpha(b_k-b^{*}_k)$ or $l(k)=\alpha(b_k-b^{*}_k)+1$:
$$\begin{pmatrix}
    ... & 1 & 0 & 0 &... & 0 & \star & 0 & ...
\end{pmatrix}.$$


If some of $b_k$'s is less than $\frac{1}{\alpha}$, then entries $1$ is replaced by the corresponding sum $\sum \limits_{k=0}^la_ke^{-2\pi w b_k}$, since in this case the place of a $\star$ in the matrix coincides with the place of the entry 1.

\medskip

From $w<0$ and $b^{*}_k<\frac{1}{\alpha}$ we conclude that
$$\max(e^{-2\pi wb^{*}_k}, e^{-2\pi w(b^*_k-\alpha^{-1})})\leq e^{2\pi |w|b_k},$$
and by assumption of theorem we have
$$\sum_{k=1}^{\infty} |a_k|e^{2\pi |w|b_k}<1.$$
Hence, matrix $\mathcal{L}_\xi$ is diagonal-dominant matrix and we get the result.

\medskip

Now we consider the case of arbitrary $\beta>0$. The condition $\sum_{k=1}^\infty|a_k|e^{2\pi |w_k|b_k}<1$
 is invariant with respect to re-normalization $(aw,b_k\slash a), a>0$. Hence, the results holds for all $\alpha\beta\leq1$.
 \qed
 
 \section{Sinc function \label{SF}}

 In this section we prove Theorem \ref{mainth2}.
 \smallskip
 Since frame set is the same for window $e^{2\pi i \omega_0 t}g(t)$ we can assume that
 $$g(t)=\frac{1-e^{2\pi w b}e^{2\pi i bt}}{t-iw}.$$
 Assume for a while what $\beta=1$. 
 
 {\bf Case $b\geq\frac{1}{\alpha}$}. Let us consider matrix operator $\mathcal{L}'_\xi$, see \eqref{lseq},
 $$\mathcal{L}'_{\xi}=\begin{pmatrix} & & ... & &\\
... & 1 & 0 &\star& 0 &0 &..\\
... & 0 & 1 & \star & 0 &...\\
... & 0 & 0 & 1 & \star & 0 &...\\
... & 0 & 0 & 0 &  1 & 0 &...\\
\end{pmatrix}.
$$
Each row of this matrix has only two non-zero entries, namely either
$\{1, -e^{2\pi w(b-b^{*})}\}$ or $\{1, -e^{2\pi w(b-b^{*}+\alpha^{-1})}\}$. We also have $w<0$. Hence the matrix  $\mathcal{L}'_\xi$ is diagonally dominated.

\smallskip

{\bf Case $b<\frac{1}{\alpha}$}. In this case we have to consider matrix $\mathcal{L}_\xi$. Each row of this matrix either has zero form $(\dots 0,0 \dots)$ or has the form $(\dots 0,1,-e^{2\pi w \slash\alpha}),0 \dots)$: the first possibility occurs when 1 and $\star$ in the matrix are on the same place, the second possibility is just a regular case. So operator $\mathcal{L}_\xi$ is bounded from below if and only if there is no trivial column in the matrix. Let us consider covering of the real line by intervals $[m, m+\alpha^{-1}]$, where $m\in\mathbb{Z}$. We have zero column with index $l$ in the matrix if $t\geq b$ for any representation $\xi+\frac{l}{\alpha}=t+m$, where $0\leq t<\frac{1}{\alpha}$, $m\in\mathbb{Z}$. So operators $\mathcal{L}_\xi$ are uniformly bounded from below if and only $b\geq 1$, since in opposite case one can construct function located in the intervals of the form $(b, 1)$ which obviously fails frame inequality.

\medskip
It remains to consider the case of arbitrary $\beta$. Usual re-normalization gives us the condition $\frac{b}{\beta}\geq1$. \qed

\section{Frame operators \label{FO}}

Let us now give an explicit formula for frame operator $S^{\alpha, \beta}$ corresponding to window function $g$ 
of the form $P(t)\slash(t-iw)$, where $P$ is a trigonometric polynomial. In this situation we are able to represent $S^{\alpha, \beta}$ as a finite sum of multiplications and convolutions.

\smallskip

For a fixed parameters $\alpha$ and $\beta$ we denote frame operator as $S$ and consider its action only on the set $C^{\infty}_0(\R)$, which is dense in $L^2(\R)$. Recall that 
$$S f = \sum_{m,n} (f, g_{m,n}) g_{m,n}.$$
Now observe that
$$g_{m,n}(x) = e^{2\pi \beta m x} {{\sum_{k = 0}^{N-1}} a_k e^{2\pi i b_k (x - \alpha n)} \over {x-iw-\alpha n}},$$
and for a smoothly enough function $f(t)$ we have
$$S f(x) = \int_{\R} \sum_{m,n} f(t) e^{-2\pi i m \beta t} {{\sum_{k = 0}^{N-1}} \overline{a}_k e^{-2\pi i b_k (t - \alpha n)} \over {t+iw-\alpha n}} e^{2\pi i m \beta x} {{\sum_{l = 0}^{N-1}} a_l e^{2\pi i b_l (x - \alpha n)} \over {x-iw-\alpha n}} \ dt =$$
$$= \int_{\R} \sum_{m,n} f(t) e^{2\pi i m \beta (x-t)} {{\sum_{k, l = 0}^{N-1}} \overline{a}_k a_l e^{2\pi i (b_l x - b_k t)} e^{2\pi i (b_k - b_l) \alpha n}\over {(t+iw-\alpha n)(x-iw-\alpha n)}} \ dt .$$
Now using the decomposition 
$${{1} \over {(t+iw-\alpha n)(x-iw-\alpha n)}} = {{1} \over {x - t - 2 i w}} \cdot \bigg( {{1} \over {t + iw - \alpha n}} - {{1} \over {x - iw - \alpha n}} \bigg)$$
one can obtain 
$$ S f(x) = \int_{\R} \sum_{m} {{f(t) e^{2\pi i m \beta (x-t)}} \over {x - t - 2 i w}} \sum_{k, l = 0}^{N-1} \overline{a}_k a_l e^{2\pi i (b_l x - b_k t)} \cdot \sum_n \bigg( {{ e^{2\pi i (b_k - b_l) \alpha n}} \over {t+iw-\alpha n}} - {{ e^{2\pi i (b_k - b_l) \alpha n }} \over {x - iw - \alpha n}} \bigg) \ dt.$$

\begin{lemma} For $C\in[0,1)$ we have
    
    $$\sum_{n\in\mathbb{Z}}\biggl{(}\frac{e^{2\pi i Cn}}{z_1-n}-\frac{e^{2\pi i Cn}}{z_2-n}\biggr{)}=2\pi i\frac{e^{2\pi i Cz_1}}{e^{2\pi i z_1}-1}-2\pi i\frac{e^{2\pi i Cz_2}}{e^{2\pi i z_2}-1}.$$
    
    \label{clemma}
\end{lemma}
\begin{proof}
The difference between right-hand side and left-hand side is an entire function $F$ with respect to variable $z_1$. On the other hand, the left-hand side is equal to $(z_2-z_1)\sum_ne^{2\pi i Cn}\slash((z_1-n)(z_2-n))$ and, hence, it is bounded by $C_1|z_1|+C_2$ on $\mathbb{C}\setminus\bigcup_n\{z: |z-n|<1\slash10\}$. So, $F$ is a linear function. In addition $F$ is bounded on the imaginary axis. Hence, $F$ is a constant. Substituting $z_1=z_2$ we get the result. 
\end{proof}
Denote by $\{x\}$ the fractional part of a real number $x$.

Using Lemma \ref{clemma} we obtain
$$S f(x) = {{2\pi i} \over {\alpha}} \int_{\R} \sum_{m} {{f(t) e^{2\pi i m \beta (x-t)}} \over {x - t - 2 i w}} \sum_{k, l = 0}^{N-1} \overline{a}_k a_l e^{2\pi i (b_l x - b_k t)} \cdot \bigg( {{e^{2\pi i \{(b_k - b_l) \alpha\} \cdot \frac{t+iw}{\alpha}}} \over {e^{2\pi i \cdot \frac{t+iw}{\alpha}} - 1}} - {{e^{2\pi i \{(b_k - b_l) \alpha\} \cdot \frac{x - iw}{\alpha}}} \over {e^{2\pi i \cdot \frac{x - iw}{\alpha}} - 1}}  \bigg) \ dt. $$

Denote by $h(x)$ the distribution $\frac{1}{x- 2 i w}\sum_{m}  e^{2\pi i m \beta x} $. Using the formula $\{-x\} = 1 - \{x\}$ one can show that
\begin{equation}
S f(x) =  {{2\pi i} \over {\alpha}} \int_{\R} f(t) h(x-t) \sum_{k = 0}^{N-1}  (I_1(k) + I_2(k) + I_3(k)) \ dt, 
\label{pfr}
\end{equation}
where $$I_1(k) = \sum_{b_k > b_l} \overline{a}_k a_l e^{2\pi i (b_l x - b_k t)} \cdot \bigg({{e^{2\pi i \{ (b_k - b_l) \alpha \} \cdot \frac{t+iw}{\alpha} }} \over {e^{2\pi i\frac{t+iw}{\alpha} } - 1}} - {{e^{2\pi i \{ (b_k - b_l) \alpha \} \cdot \frac{x-iw}{\alpha} }} \over {e^{2\pi i\frac{x-iw}{\alpha} } - 1}} \bigg),$$

 $$I_2(k) = \sum_{b_k < b_l} \overline{a}_k a_l e^{2\pi i (b_l x - b_k t)} \cdot \bigg( 
 {{e^{-2\pi i \{ (b_l - b_k) \alpha \} \cdot \frac{t+iw}{\alpha} }} \over {1-e^{-2\pi i\frac{t+iw}{\alpha}}}} - {{e^{-2\pi i \{ (b_k - b_l) \alpha \} \cdot \frac{x-iw}{\alpha} }} \over {1-e^{-2\pi i\frac{x-iw}{\alpha} }}} \bigg),$$
and 
$$I_3(k) = \sum_{k} |a_k|^2 e^{2\pi i b_k (x - t)} \cdot \bigg( 
 {{1} \over {e^{2\pi i\frac{t+iw}{\alpha}}-1}} - {{1} \over {e^{2\pi i\frac{x-iw}{\alpha} } -1}} \bigg).$$

\subsection{Partial case: small parameters}
Let $\alpha \leq {{1 } \over {\max_{0 \leq k,l \leq N-1}} |b_k - b_l|}$. Then we have 
$$\{(b_k - b_l) \alpha \} = \begin{cases}
    (b_k - b_l) \alpha, & b_k > b_l, \\
    1 - (b_l - b_k) \alpha, & b_k < b_l, \\
    0, & b_k = b_l.
\end{cases}$$
 Now one can rewrite formulae as
 $$ I_1(k) = \sum_{b_k > b_l} \overline{a}_k a_l \bigg( \frac{e^{2\pi i (b_l x - b_k t)} e^{2\pi i (b_k - b_l) (t + iw)}}{e^{2\pi i \frac{t+iw}{\alpha}}-1}  - \frac{e^{2\pi i (b_l x - b_k t)} e^{2\pi i (b_k - b_l) (x - iw)}}{e^{2\pi i \frac{x-iw}{\alpha}}-1} \bigg),$$

 $$ I_2(k) = \sum_{b_k < b_l} \overline{a}_k a_l \bigg( \frac{e^{2\pi i (b_l x - b_k t)} e^{2\pi i \frac{t+iw}{\alpha}} e^{2\pi i (b_k - b_l) (t + iw)}}{e^{2\pi i \frac{t+iw}{\alpha}}-1}  - \frac{e^{2\pi i (b_l x - b_k t)} e^{2\pi i \frac{x-iw}{\alpha}} e^{2\pi i (b_k - b_l) (x - iw)}}{e^{2\pi i \frac{x-iw}{\alpha}}-1} \bigg).$$

Put
$$h_1^{(1)}(s) = h(s) \sum_{k=0}^{N-1} \sum_{b_k > b_l} \overline{a}_k a_l e^{2\pi i b_l s} e^{-2\pi w(b_k-b_l)},$$
 $$h_2^{(1)}(s) = h(s) \sum_{k=0}^{N-1} \sum_{b_k < b_l} \overline{a}_k a_l e^{2\pi i b_l s} e^{-2\pi w(b_k-b_l)},$$
 $$h_1^{(2)}(s) = h(s) \sum_{k=0}^{N-1} \sum_{b_k > b_l} \overline{a}_k a_l e^{2\pi i b_k s} e^{2\pi w(b_k-b_l)},$$
 $$h_2^{(2)}(s) = h(s) \sum_{k=0}^{N-1} \sum_{b_k < b_l} \overline{a}_k a_l e^{2\pi i b_k s} e^{2\pi w(b_k-b_l)}.$$
 Applying formulae given above to the definition of $Sf$ one can obtain the following result.
 
 \begin{theorem}
    Let $\alpha \cdot \max\limits_{k,l}|b_k - b_l| \leq 1$, $f \in L^2(\R)$ and functions $h(z), h_i^{(j)}(z)$ be defined above. Then we have
$$ {{\alpha} \over {2\pi i}} Sf(x) = \left( \frac{f(t)}{e^{2\pi i \frac{t+iw}{\alpha}} -1} * h_1^{(1)}\right) - \frac{1}{e^{2\pi i \frac{x-iw}{\alpha}} -1} \left( f* h_2^{(1)}\right) +$$
$$+\left( \frac{f(t)}{1-e^{-2\pi i \frac{t+iw}{\alpha}}} * h_1^{(2)}\right) - \frac{1}{1-e^{-2\pi i \frac{x-iw}{\alpha}}} \left( f* h_2^{(2)}\right) + $$
$$+ \sum_{k=0}^{\infty} \sum_{b_k = b_l} |a_k|^2 e^{2\pi i b_k x} \bigg(\left( \frac{f(t) e^{-2\pi i b_k t}}{e^{2\pi i \frac{t+iw}{\alpha}}-1} * h(s) \right) - \frac{1}{e^{2\pi i \frac{x-iw}{\alpha}}-1} \left( f(t) e^{-2\pi i b_k t } * h(s) \right) \bigg). $$

 \end{theorem}

\subsection{Sinc-function}
Consider $g(z) = \frac{\sin \pi (z - iw)}{\pi (z - iw)}$. Then we need to examine a trigonometric polynomial $P(t) = 1 - e^{2\pi w} e^{2\pi i t}$ and a function 
$$g_{m,n}(x) = e^{2\pi i m \beta x} \cdot \frac{1 - e^{2\pi w} e^{2\pi i(x - \alpha n)}}{x - iw - \alpha n}.$$
In this case we obtain 

$$Sf(x) = \int_{\R} \sum_m \frac{f(t)e^{2\pi i m \beta (x-t)}}{x-t-2i w} \cdot \bigg( (1+e^{4\pi w} e^{2\pi i (x-t)}) \sum_n \left( \frac{1}{t + iw - \alpha n} - \frac{1}{x-iw-\alpha n} \right) 
$$
\begin{equation}
- e^{2\pi w} \sum_n \bigg( \frac{e^{-2\pi i (t - \alpha n)} + e^{2\pi i(x - \alpha n)}}{t + iw - \alpha n} - \frac{e^{-2\pi i (t - \alpha n)} + e^{2\pi i(x - \alpha n)}}{x-iw - \alpha n} \bigg) \bigg)\ dt.
\end{equation}
Now using similar technique and expressing the sums one can show that
$$Sf(x) = S_R f(x) + {{2\pi i} \over {\alpha}} e^{2\pi w} \int_{\R} \sum_m \frac{f(t)e^{2\pi i m \beta (x-t)}}{x - t - 2iw} \bigg( e^{-2\pi i t} \bigg( \frac{e^{2\pi i \{\alpha \} \frac{x-iw}{\alpha}}}{e^{2\pi i \frac{x-iw}{\alpha}} -1} -\frac{e^{2\pi i \{\alpha \} \frac{t+iw}{\alpha}}}{e^{2\pi i \frac{t+iw}{\alpha}} -1} \bigg) +$$
\begin{equation}
+e^{2\pi i x} \bigg( \frac{e^{-2\pi i \{\alpha \} \frac{x-iw}{\alpha}}}{1-e^{-2\pi i \frac{x-iw}{\alpha}}} -\frac{-e^{2\pi i \{\alpha \} \frac{t+iw}{\alpha}}}{1-e^{-2\pi i \frac{t+iw}{\alpha}}} \bigg) \bigg) \ dt,
\label{sincfr}
\end{equation}
where the operator $S_R f(x)$ can be represented using similar technique as \cite[page 8]{BKL2}:
$$S_R f(x) = {{\pi} \over {\alpha}} \sin^{-1} \left( {{\pi} \over {\alpha}} (x -iw) \right) \left( f(t) \sin^{-1} \left( {{\pi} \over {\alpha}} (t+ iw) \right)  * h_1\right),$$
where $$h_1(s) = \sum_m e^{2\pi i \beta m s} \frac{\sin \left( \frac{\pi}{\alpha} (s - 2iw) \right) }{s-2iw} \cdot (1+ e^{4\pi w} e^{2\pi i s}).$$

 \subsection{Sinc-function with $\alpha < 1$}
In the case of $\alpha < 1$ one can express frame operator in a more simpler form. Let 
$$h_3(x) = \frac{-1+e^{4\pi w} e^{2\pi i x}}{x-2iw}\sum_m e^{2\pi i m \beta x}.$$
\begin{theorem}
    Let $g(x) = \frac{1-e^{2\pi w} e^{2\pi i x}}{x - iw}$ and $\alpha <1$. Then $Sf(x) = \frac{2\pi i}{\alpha} f * h_3.$
\end{theorem}
\begin{proof}
    In this case we have 
    $$\frac{\alpha}{2\pi i} Sf(x) = \int_{\R} \sum_m f(t) \frac{e^{2\pi i m \beta (x-t)}}{x-t-2iw} \bigg( (1+e^{4\pi w} e^{2\pi i (x-t)} ) \cdot \bigg( \frac{1}{e^{2\pi i \frac{t+iw}{\alpha}}-1} - \frac{1}{e^{2\pi i \frac{x-iw}{\alpha}}-1}\bigg) -$$
    $$- e^{2\pi w} \cdot \bigg( \frac{e^{-2\pi w}}{e^{2\pi i \frac{t+iw}{\alpha}} -1} + \frac{e^{2\pi i x} e^{-2\pi i (t+iw)} e^{2\pi i \frac{t+iw}{\alpha}}}{e^{2\pi i \frac{t+iw}{\alpha}} -1} \bigg) +$$
    $$+e^{2\pi w} \cdot\bigg( \frac{e^{-2\pi i t} e^{2\pi i x} e^{2\pi w}}{e^{2\pi i \frac{x-iw}{\alpha}} -1} + \frac{e^{2\pi i x} e^{-2\pi i (x-iw)} e^{2\pi i \frac{x-iw}{\alpha}}}{e^{2\pi i \frac{x-iw}{\alpha}} -1}  \bigg) \ dt.$$

Now sum of the items belonging to $\frac{1}{e^{2\pi i \frac{t+iw}{\alpha}} -1}$ equals to 
$$e^{4\pi w} e^{2\pi i (x-t)} (1 - e^{2\pi i \frac{t+iw}{\alpha}}),$$
and sum of the items belonging to $\frac{1}{e^{2\pi i \frac{x-iw}{\alpha}} -1}$ equals to $-1 + e^{2\pi i \frac{x-iw }{\alpha}}$. So we obtain
$$\frac{\alpha}{2\pi i} Sf(x) = \int_{\R} \sum_m f(t) \frac{e^{2\pi i m \beta (x-t)}}{x-t-2iw} (e^{4\pi w} e^{2\pi i (x-t)} -1) \ dt.$$

 \end{proof}

\subsection*{Acknowledgments}
We express our gratitude to Beijing Institute of Mathematical Sciences and Applications for hospitality.


\begin{thebibliography}{9999}

\bibitem{BKL1} Y. Belov, A. Kulikov, Y. Lyubarskii, \textit{Gabor frames for rational functions}, Inventiones Mathematique, 231:431--466 (2023).


\bibitem{BC} K. Bittner and C. K. Chui, {\em Gabor frames with arbitrary windows}, In C. K. Chui, L. L. Schumaker, and J. St\"{o}ckler, editors, 
in Approximation theory, X (St. Louis, MO, 2001), Innov. Appl. Math., pp. 41–50, Vanderbilt Univ. Press, Nashville, TN, 2002.

\bibitem{DS} X.~Dai, Q.~Sun, {\em The $abc$-problem for Gabor systems,} Memoirs of the American Mathematical Society,
{\bf 244}, 1152, (2016).

\bibitem{Gro} K.~Gr\"ochenig, {\em  Foundations of Time-Frequency Analysis,} Birkh\"auser, Boston, MA, 2001.


\bibitem{Gro2} K.~Gr\"ochenig, J.~St\"ockler, {\em Gabor frames and totally positive functions,}
Duke Mathematical Journal, 162 (6), 1003--1031, (2011).

\bibitem{Gro1}K.~Gr\"ochenig, J.L.~Romero, J.~St\"ockler, {\em Sampling theorems for shift-invariant spaces, Gabor frames, and totally positive functions,}
Inventiones mathematicae, 211 (3), 1119--1148, (2016).

Appl., {\bf 13}(2):113--166, (2007).

\bibitem{Jans}   A. J. E. M.~Janssen, {\em Zak transforms with few zeros and the tie,}  in Advances in Gabor analysis, pp. 31--70, Appl. Numer. Harmon. Anal., Birkhäuser Boston, (2003).

\bibitem{Jans2} A. J. E. M.~Janssen, {\em Some Weyl-Heisenberg frame bound calculations,} Indag. Math., 7:165--
182, (1996).

\bibitem{Jans3} A. J. E. M.~Janssen, {\em On generating tight Gabor frames at critical density,} J. Fourier Anal.
Appl., 9(2):175--214, (2003).

\bibitem{Jans4} A. J. E. M.~Janssen, {\em Some counterexamples in the theory of Weyl-Heisenberg frames},
 IEEE Trans. Inform. Theory, 42(2):621--623, (1996).

\bibitem{JansStr} A.~Janssen, T.~Strohmer, {\em Hyperbolic secants yield Gabor frames},
Appl. Comput. Harmon. Anal., 12, 259--267, (2002).

\bibitem{L} Yu.~Lyubarskii, {\em Frames in the Bargmann space of entire functions,} in: Entire and Subharmonic Functions, Adv. Soviet Math., vol. 11, Amer.\ Math.\ Soc., Providence, RI, 1992, pp. 167--180.

\bibitem{S} K.~Seip, {\em Density theorems for sampling and interpolation in the Bargmann--Fock space. I,} J. Reine Angew.\ Math.\ {\bf 429} (1992) 91--106.

\bibitem{SW} 
K.~Seip, R.~Wallst\'en, {\em Density theorems for sampling and interpolation in the Bargmann--Fock space. II, }J. Reine Angew.\ Math.\ {\bf 429}  (1992) 107--113.

\end{thebibliography}
\end{document}